\documentclass[a4paper]{amsart}

  \usepackage{geometry}
\usepackage{hyperref}
\usepackage{todonotes}
\usepackage{pgffor,etex,tasks,yhmath}
 
\usepackage{amsmath,amsthm,mathrsfs} 

\usepackage{amssymb} 

\usepackage{amscd} 

\usepackage[all]{xy} 

\usepackage{color} 

\usepackage{combelow} 

\usepackage{enumerate}

\newtheorem{theorem}{Theorem}
\newtheorem{lemma}{Lemma}

\newtheorem*{theoremA}{Theorem A}
\newtheorem*{theoremB}{Theorem B}
\newtheorem*{theoremC}{Theorem C}
\newtheorem*{theoremD}{Theorem D}
\newtheorem*{theoremE}{Theorem E}
\newtheorem*{theorem*}{Theorem}

\newtheorem{corollary}{Corollary}
\newtheorem*{corollary*}{Corollary}
\newtheorem{definition}{Definition}
\newtheorem{example}{Example}
\newtheorem{proposition}{Proposition}
\newtheorem{remark}{Remark}

\newcommand{\diffto}{\xrightarrow{\raisebox{-0.2 em}[0pt][0pt]{\smash{\ensuremath{\sim}}}}}

\newcommand{\inc}{\hookrightarrow}
\newcommand{\rmap}{\longrightarrow}

\newcommand{\acts}{\curvearrowright}

\newcommand{\hor}{\mathrm{hor}}
\newcommand{\R}{\mathbb{R}}

\newcommand{\bd}{\partial}

\newcommand{\pr}{\operatorname{pr}}
\newcommand{\dd}{\mathrm{d}}

\newcommand{\HH}{\mathrm{H}}

\newcommand{\id}{\mathrm{id}}

\newcommand{\8}{\infty}

\newcommand{\N}{\mathbb{N}}
\newcommand{\X}{\mathfrak{X}}

\begin{document}
\title{Submersions by Lie algebroids}
\author{Pedro Frejlich}
\address{UFRGS, Instituto de Matem\'atica e Estat\'istica, Porto Alegre, Brasil}
\email{frejlich.math@gmail.com}
\maketitle
\begin{abstract}
 In this note, we examine the bundle picture of the pullback construction of Lie algebroids. The notion of submersions by Lie algebroids is introduced, which leads to a new proof of the local normal form for lie algebroid transversals of \cite{BLM}, and which we use to deduce that Lie algebroids transversals concentrate all local cohomology.
 
 The locally trivial version of submersions by Lie algebroids $\mathfrak{S}$ is then discussed, and we show that this notion is equivalent to the existence of a complete Ehresmann connection for $\mathfrak{S}$, extending the main result in \cite{Matias}.
 
 Finally, we show that locally trivial version of submersions by Lie algebroids gives rise to a system of local coefficients, which is an integral part of a version of the homotopy invariance of de Rham cohomology in the context of Lie algebroids, and we apply such local systems to extend the localization theorem of \cite{ChenLiu_Loc}.
\end{abstract}
 \tableofcontents

\section{Introduction}

Foremost among the properties that general Lie algebroids share with the prototypical example of tangent bundles is that they enjoy a wholly analogous machinery of \emph{flows}: every section $a \in \Gamma(A)$ of a Lie algebroid $A$ on $M$ gives rise (modulo completeness issues) to a morphism Lie algebroids
\[
 (\Phi,\phi):A \times T\R \rmap A,
\]its flow, uniquely characterized by $\bd_{\epsilon}\Phi_{\epsilon}^{\dagger}(b)=\Phi_{\epsilon}^{\dagger}([a,b]_A)$ for all $b \in \Gamma(A)$, where $\Phi_{\epsilon}^{\dagger}(b)$ denotes $\Phi_{\epsilon}^{-1}(b \circ \phi)$. Already for the case of tangent bundles, submanifolds transverse to flows play a certain privileged role of \emph{local trivialization}. This phenomenon has its clearest illustration in the collaring neighborhood theorem for manifolds with boundary \cite[Corollary 3.5]{Milnor}, and lies at the heart of the homotopy classification of bundles.

The present note is concerned with the interaction between transverse submanifolds and flows, in which strong emphasis is put on functorial and cohomological aspects. Concretely, we study smooth families of Lie algebroids $A_x$, parametrized by a smooth manifold $x \in M$. Smoothness is encoded into the existence of a surjective submersion $p:\Sigma \to M$, and a Lie algebroid $K$ on $\Sigma$, whose anchor $\varrho_K:K \to T\Sigma$ is vertical, so that $K$ restricts on each fibre $\Sigma_x:=p^{-1}(x)$ to a Lie algebroid $A_x$.

If we wish to compare points in arbitrary different fibres using Lie algebroid flows, the natural condition to impose is that there be an extension of $TM$ by $K$ -- that is, that there be a Lie algebroid $A$ on $\Sigma$, in which $K$ sits as the kernel of a surjective morphism of Lie algebroids $A \to TM$ over $p$. This condition is equivalent to demanding that each fibre of $p$ sit in $A$ as a \emph{Lie algebroid transversal}. This is abstracted into the main notion introduced in this note:

\begin{definition}
A {\bf submersion by Lie algebroids} $\mathfrak{S}=(A,p)$ on $M$ is a Lie algebroid $A$ on the total space of a surjective submersion $p:\Sigma \to M$, in which each fibre is a Lie algebroid transversal.  
\end{definition}

There is an obvious notion of morphism between submersions by Lie algebroids, as well as an obvious notion of Ehresmann connection for such objects. A simple, but crucial property enjoyed by submersions by Lie algebroids $\mathfrak{S}=(A,p)$ on $M$ is that they pull back under arbitrary smooth maps $\phi:N \to M$, in the sense that there is an induced submersion by Lie algebroids $\phi^!\mathfrak{S}$, together with a morphism into $\mathfrak{S}$.

Submersions by Lie algebroids specialize in extreme cases to surjective submersion (when $A=T\Sigma$), transitive Lie algebroids (when $\Sigma=M$ and $p=\id_M$), and to Lie algebroids themselves (when $M$ is a point). Other examples arise by pulling back Lie algebroids by map from product manifolds $N \times M$, in which transversality is demanded of each fixed $x \in M$. For instance, a Lie algebroid $A$ on the total space of a vector bundle $p:E \to X$ pulls back under the scalar multiplication map $m:E \times [0,1] \to E$ to a submersion by Lie algebroids over the unit interval iff $X$ is a Lie algebroid transversal in $A$. This observation yields another proof of the normal form for transversals in Lie algebroids \cite{BLM}; in both Theorems A and B below, $A$ be a Lie algebroid on $M$, $\nabla:A \curvearrowright D$ a representation, and $i:X\inc M$ the inclusion of a Lie algebroid transversal.
\begin{theoremA}
On a tubular neighborhood $p:U \to X$ of $X$ in $M$, the representations $p^!i^!(\nabla)$ and $\nabla$ are isomorphic,
\end{theoremA}
\begin{theoremB}
The inclusion $i:X \inc U$ induces an isomorphism in Lie algebroid cohomology with coefficients: $i^*:\HH(A|_U;D|_U) \diffto \HH(i^!(A);i^*(D))$.
\end{theoremB}
\noindent See Theorem \ref{thm : normal forms for transversals}, Corollary \ref{cor : nf for reps} and Theorem \ref{thm : cohomology concentrated} for detailed statements.

\vspace{0.1cm}

In any case, it follows that each 'Lie algebroid fibre' $i_x^!\mathfrak{S}$ of a submersion by Lie algebroids $\mathfrak{S}=(A,p)$ determines the ambient Lie algebroid structure in an open neighborhood of the 'manifold fibre' $\Sigma_x=p^{-1}(x)$. However, such a neighborhood need not be saturated by fibres of $p$, and this completeness issue leads naturally to the following

\begin{definition}
 A submersion by Lie algebroids $\mathfrak{S}$ on $M$ is {\bf locally trivial} if every point $x \in M$ has an open neihborhood $U$, over which there is an isomorphism $i_x^!\mathfrak{S} \times TU \diffto i_U^!\mathfrak{S}$ covering $\id_U$.
\end{definition}For example, transitive Lie algebroids automatically satisfy this condition. Our next result generalizes \cite{Matias} to give a characterization of locally trivial submersions by Lie algebroids:
\begin{theoremC}
 A submersion by Lie algebroids $\mathfrak{S}$ is locally trivial if and only if it admits a complete Ehresmann connection.
\end{theoremC}

For locally trivial submersions by Lie algebroids, $\mathfrak{S}=(A,p)$ and any representation $\nabla:A \acts D$, the assignment
\[
\mathscr{H}_{\nabla}(U):=\HH(A|_{\Sigma|_U},D|_{\Sigma|_U}) 
\]defines a locally constant presheaf on $M$, so the associated sheaf $\pr:\mathscr{H}_{\nabla}^+ \to M$ is a covering space. In the next result, we describe the monodromy of the covering:

\begin{theoremD}
 The monodromy of $\mathscr{H}_{\nabla}$ with respect to a path $c:[0,1] \to M$ can be described by either:
 \begin{enumerate}[i)]
  \item parallel transport along $c$ with respect to any complete Ehresmann connection;
  \item extensions to closed cocycles on $c^!\mathfrak{S}$,
 \end{enumerate}
\end{theoremD}\noindent see Theorem \ref{thm : Monodromy} for a detailed statement. When the 'Lie algebroid fibre' has finite-dimensional cohomology, $\mathscr{H}_{\nabla}^+$ is the sheaf of local parallel sections of a flat vector bundle $\mathbb{H}$ on $M$. We conclude by generalizing to locally trivial submersions by Lie algebroids the localization principle for cohomology classes of degree one given in \cite[Theorem 2.2]{ChenLiu_Loc}:

\begin{theoremE}
 If $\mathfrak{S}=(A,p)$ is a locally trivial submersion by Lie algebroids on a connected manifold $M$, and $\nabla:A \acts D$ is a representation for which $\HH^q(A_x;D_x)=0$ for $q<n-1$. Then the restriction map\[i_x^*:\HH^n(A;D) \to \HH^n(A_x;D_x)\]is injective, provided that either $\HH^{n-1}(A_x;D_x)=0$, or that $M$ be simply connected and $\HH^{n-1}(A_x;D_x)$ have finite dimension.
\end{theoremE}\noindent See Theorem \ref{thm : localization of cohomology}.

\vspace{0.2cm}

The paper is organized as follows:

\vspace*{0.1cm}

\noindent \underline{Section \ref{sec : Submersions by Lie algebroids}} introduces our main object of study, \emph{submersions by Lie algebroids}. We discuss their basic structure and morphisms, and the crucial fact for what follows -- namely, that they pull back under any smooth map. It is there that we prove Theorems A and B.

\vspace*{0.1cm}

\noindent The notions of local triviality and Ehresmann connections are discussed in \underline{Section \ref{sec : Local triviality and completeness}}, where we prove Theorem C, characterizing locally trivial submersions by Lie algebroids as those which admit complete Ehresmann connections. This leads to the system of local coefficients of \underline{Section \ref{sec : The system of local coefficients}}, where we prove Theorem D concerning its monodromy. 

\vspace*{0.1cm}

\noindent Finally, in \underline{Section \ref{sec : Localization}}, in which we discuss a spectral sequence naturally associated with locally trivial submersions by Lie algebroids, and which leads to Theorem E concerning the localization of cohomology classes to fibres.

\vspace{0.1cm}

\subsection*{Acknowledgement}
Work partially supported by the Nederlandse Organisatie voor Wetenschappelijk Onderzoek (Vrije Competitie grant ``Flexibility and Rigidity of Geometric Structures'' 612.001.101) and by IMPA (CAPES-FORTAL project). I would like to thank Olivier Brahic and Ori Yudilevich for the many improvements they suggested, and David Mart\'inez-Torres and Rui Loja Fernandes for many interesting conversations. Above all, I would like to thank Ioan ~M\u{a}rcu\cb{t}, in discussions with whom most of the present project was conceived and developed.

\section{Submersions by Lie algebroids}\label{sec : Submersions by Lie algebroids}

Recall from Definition 1 that a {\bf submersion by Lie algebroids} $\mathfrak{S}=(A,p)$ consists of a Lie algebroid $A$ on the total space of a surjective submersion $p:\Sigma\to M$, with the property that each fiber $\Sigma_x:=p^{-1}(x)$ is transverse to $A$.


\begin{lemma}
Let $p:\Sigma\to M$ be a surjective submersion, $V(\Sigma):=\ker p_*$ its vertical bundle, and $A$ a Lie algebroid on $\Sigma$. Then the induced morphism of Lie algebroids $(\tau_A,p) : A \to TM$, $\tau_A=p_*\varrho_A$ is surjective exactly when $\mathfrak{S}=(A,p)$ is a submersion by Lie algebroids, in which case $K(A):=\ker \tau_A$ is a Lie subalgebroid of $A$ with vertical anchor.
\end{lemma}
\begin{proof}
$(\tau_A,p) : A \to TM$ is a morphism of Lie algebroids since it is the composition of morphisms of Lie algebroids $(p_*,p):T\Sigma \to TM$ and $(\varrho_A,\id):A \to T\Sigma$, and it is surjective exactly when $V(\Sigma) + \varrho_A(A) = T\Sigma$, i.e., when $\mathfrak{S}=(A,p)$ is a submersion by Lie algebroids. This implies that $K(A)=\ker \tau_A$ is a Lie subalgebroid, and that $\varrho_A:K(A) \to T\Sigma$ factors through $V(\Sigma) \subset T\Sigma$.
\end{proof}

A submersion by Lie algebroids $\mathfrak{S}=(A,p)$ thus induces an exact sequence of vector bundles over $\Sigma$:
\begin{align*}
0 \rmap K(A) \rmap  A \stackrel{\tau_A}{\rmap} p^*(TM) \rmap 0
\end{align*}

\begin{lemma}\label{lem : path and bundles of lalgs}
 Let $\phi:\Sigma \to N$ a smooth map transverse to a Lie algebroid $B$ on $N$. Then the following conditions are equivalent:
 \begin{enumerate}[i)]
  \item $\phi_x:\Sigma_x \to N$ is transverse to $B$ for all $x \in M$;
  \item $A:=\phi^!(B)$ turns $p:\Sigma \to M$ into a submersion by Lie algebroids $\mathfrak{S}(\phi,B):=(A,p)$.
 \end{enumerate}
\end{lemma}
\begin{proof}
 Let $x=p(y) \in M$ and let $u \in T_y\Sigma$. Assuming i), there exist $v \in V(\Sigma)_y$ and $b \in B_{\phi_(y)}$, such that
 \[
  \phi_*(u)=\phi_{x*}(v)+\varrho_{B}(b).
 \]This implies that $a:=(u-v,b) \in A_y$, and that
 \[
  a:=(u-v,b) \in A_{y}, \quad u = \varrho_{A}(a) + (v,0) \in \varrho_{A}(A) + V(\Sigma),
 \]which shows that ii) holds. Conversely, assume ii) and denote by $(\Phi,\phi):\phi^!(A) \to A$ the pullback morphism of Lie algebroids. Because $\phi$ is transverse to $B$, for each $w \in T_{\phi(y)}N$, there exist $b \in B_{\phi(y)}$ and $u \in T_{y}\Sigma$ such that \[w = \phi_*(u)+\varrho_{B}(b).\]Because $\mathfrak{S}=(\phi^!(B),p)$ is a submersion by Lie algebroids, there exist $a \in A_{y}$ and $v \in V(\Sigma)_y$ such that\[u = \varrho_{A}(a) + v.\]This implies that
 \begin{align*}
  w & = \phi_*(\varrho_{A}(a) + v)+\varrho_{B}(b) = \phi_*(v)+\varrho_{B}(\Phi(a)+b) \\ & = \phi_{x*}(v)+\varrho_{B}(\Phi(a)+b)
 \end{align*}and so i) holds.
\end{proof}

\begin{example}
 A surjective submersion $p:\Sigma \to M$ defines a submersion by Lie algebroids $\mathfrak{S}(p,TM)$. A transitive Lie algebroid $A$ on $M$ defines a submersion by Lie algebroids $\mathfrak{S}(\id_M,A)$.
\end{example}

For later reference, we state the following simple consequence of Lemma \ref{lem : path and bundles of lalgs}:

\begin{corollary}[Transversality and rescaling]\label{cor : lalg and scalar}
 Let $p:E \to X$ be a vector bundle, with scalar multiplication map $m:E \times \mathbb{R} \to E$. Then for a Lie algebroid $A$ on $E$, the following conditions are equivalent:
 \begin{enumerate}[i)]
  \item The zero section $X \inc E$ is transverse to $A$;
  \item for all $t \in \mathbb{R}$, $m_t:E \to E$ is transverse to $A$;
  \item $m^!(A)$ turns $\mathrm{pr}:E \times \mathbb{R} \to \mathbb{R}$ into a submersion by Lie algebroids.
 \end{enumerate}
\end{corollary}
\begin{proof}
 Just note that $m_t:E \to E$ is a diffeomorphism for $t \neq 0$, so ii) holds iff $m_0$ is transverse to $A$; and since $m_{0*}(TE)=TX$, i) and ii) are equivalent, and ii) $\Leftrightarrow$ iii) by Lemma \ref{lem : path and bundles of lalgs}.
\end{proof}

A {\bf morphism} $\Phi : \mathfrak{S}_0\to \mathfrak{S}_1$ between submersion by Lie algebroids  $\mathfrak{S}_i= (A_i,p_i)$ is a morphism of Lie algebroids $(\Phi,\varphi):A_0 \to A_1$, in which $\varphi:\Sigma_0 \to \Sigma_1$ covers a smooth map $\phi:M_0 \to M_1$.

\begin{example}
 The anchor defines a morphism $\varrho_A:\mathfrak{S}=(A,p)\to\mathfrak{S}(p,p,TM)$.
\end{example}

A crucial feature of submersions by Lie algebroids is that they pull back under arbitrary smooth maps into the base.

\begin{lemma}\label{lem : Pullback submersion by Lie algebroids}
Let $\mathfrak{S}=(A,p)$ be a submersion by Lie algebroids on $M$, and for a smooth map $\phi:N \to M$, denote by $\overline{p}:\phi^*(\Sigma) \to N$ the pullback submersion, and by $\varphi:\phi^*(\Sigma) \to:\Sigma$ the pullback map. Then $\phi^!\mathfrak{S}=(\varphi^!(A),\overline{p})$ is a submersion by Lie algebroids on $N$, and the pullback morphism of Lie algebroids $(\Phi,\varphi):\varphi^!(A) \to A$ defines a pullback morphism of submersions by Lie algebroids $\Phi:\phi^!\mathfrak{S} \to \mathfrak{S}$.
\end{lemma}
\begin{proof}
Let $C \subset A$ be any vector subbundle for which $V(\Sigma)+\varrho_{A}(C) = A$. Because
\[
 \varphi_*V(\phi^*(\Sigma))_{y} = V(\Sigma)_{\varphi(e')}, \quad V(\Sigma)_{x} + \varrho_{A}(A)_{x} = T_{x}\Sigma, \quad y \in \phi^*(\Sigma), \quad x \in \Sigma
\]we have that the pullback map $\varphi:\phi^*(\Sigma) \to \Sigma$ is transverse to $C$, i.e. $\varphi_*(T_y\phi^*(\Sigma))+\varrho_A(C_{\varphi(y)})=T_{\varphi(y)}\Sigma$. In particular, $\varphi$ is transverse to $A$, whence the pullback morphism of Lie algebroids $(\Phi,\varphi):\varphi^!(A) \to A$ is defined.

We check that fibres of $\overline{p}$ are transversals in $\varphi^!(A)$: let $w \in T_{y}\phi^*(\Sigma)$, and decompose $\varphi_*(w) = v + \varrho_A(a)$ with $v \in V(\Sigma)_{\varphi(y)}$ and $a \in A_{\varphi(y)}$. Now, $v=\varphi_*(u)$ for some $u \in V(\phi^*(\Sigma))_{y}$ implies that $\varphi_*(w-u) = \varrho_A (a)$, which is to say that $a':=(w-u,a) \in \varphi^!(A)_{y}$, and therefore $w = u+\varrho_{\varphi^!(A)}(a')$. This proves that $\phi^!\mathfrak{S}=(\varphi^!(A),\overline{p})$ is a submersion by Lie algebroids, and by construction $(\Phi,\varphi,\phi):\phi^!\mathfrak{S} \to \mathfrak{S}$ is a morphism of submersions by Lie algebroids. 
\end{proof}

We will call $\phi^!\mathfrak{S}$ the {\bf pullback submersion by Lie algebroids}, and $(\Phi,\varphi,\phi):\phi^!\mathfrak{S} \to \mathfrak{S}$  the {\bf pullback morphism} of submersions by Lie algebroids. 

\noindent The next theorem gives a normal form for transversals in Lie algebroids which is much in the spirit of \cite{BLM}. 

\begin{theorem}[Normal form for Lie algebroid transversals]\label{thm : normal forms for transversals}
 If $X \subset (M,A)$ is a Lie algebroid transversal, then for every tubular neighborhood $p:E \to X$ of $X$ in $M$, there exists an open neighborhood $U \subset E$ of $X$ and a Lie algebroid isomorphism
 \[(\Phi,\varphi) : p^!i^!(A)|_U \diffto A|_{\varphi(U)}\]extending $\widetilde{i}$ on $i^!(A) \subset p^!i^!(A)$.
\end{theorem}
\begin{proof}
 It follows from Corollary \ref{cor : lalg and scalar} that $m^!(A)$ turns $\pr_2:E \times I \to I$ into a submersion by Lie algebroids, in which the fibre over $\epsilon \in I$ is canonically identified with $m_{\epsilon}^!(A)$, and $m_0^!(A)=p^!i^!(A)$, where $i:X \to E$ is the inclusion. Note also that $X \times I$ is a Lie algebroid transversal in $(E \times I,m^!(A))$, and that
 \[
  \left( X \times I,(i,\mathrm{id})^!m^!(A) \right) = (X,i^!(A)) \times (I,TI).
 \]Hence we can choose a section $a  \in \Gamma(m^!(A))$, such that
 \begin{tasks}(2)
  \task $a$ coincides with $\frac{\bd}{\bd \epsilon}$ along $X \times I$;
  \task $\tau_{m^!(A)}(a) = \frac{\bd}{\bd \epsilon}$
 \end{tasks}
Let $(\Phi_{\epsilon},\varphi_{\epsilon})$ denote the local flow of $a$ (see e.g. \cite[Appendix]{CrFer_Int_Lie}). Because of condition a), if $b \in (i,\mathrm{id})^!m^!(A)$, then $\Phi_{\epsilon}(b)$ is defined up to time $\epsilon=1$; this implies by condition b) that there is an open neighborhood $U \subset E\times \{0\}$ of $X \times \{0\}$, such that $\Phi_{\epsilon}$ induces a path of isomorphisms of Lie algebroids
\[
 (\Phi_{\epsilon},\varphi_{\epsilon}) : m_0^!(A)|_U \diffto m_{\epsilon}^!(A)|_{\varphi_{\epsilon}(U)}, \quad \epsilon \in [0,1],
\]which restricts to the identity on $i^!(A) \subset m_0^!(A)|_U = (p|_U)^!i^!(A)$. Hence at time one we get an isomorphism
\[
 (\Phi_{1},\varphi_{1}) : m_0^!(A)|_U = (p|_U)^!i^!(A) \diffto A|_{\varphi_1(U)} = m_{1}^!(A)|_{\varphi_{1}(U)}. \qedhere
\]
\end{proof}

Let now the Lie algebroid $A$ on $M$ have a representation $\nabla$ on a vector bundle $D$. Then $\mathscr{A}_{\nabla}:=A \oplus D$ has the structure of Lie algebroid on $M$, with bracket and anchor
\[
 \varrho_{\mathscr{A}_{\nabla}}(a,s)=\varrho_A(a), \quad [(a,s),(a',s')]_{\mathscr{A}_{\nabla}}=([a,a']_A,\nabla_as'-\nabla_{a'}s)),
\]and this establishes a bijection between representations of $A$ on $D$ and Lie algebroids $\mathscr{A}_{\nabla}$ in which $D$ sits as a totally intransitive, abelian subalgebroid, with quotient $A$. Then Theorem \ref{thm : normal forms for transversals} has the following immediate version for representations:

\begin{corollary}\label{cor : nf for reps}
For a Lie algebroid transversal $X \subset (M,A)$, and a representation $\nabla:A \curvearrowright D$, the representations $p^!i^!(\nabla)$ and $\nabla$ are isomorphic around $X$, by an isomorphism which is identical on $i^!(\nabla) \subset p^!i^!(\nabla)$.
\end{corollary}
\begin{proof}
As in the proof of Theorem \ref{thm : normal forms for transversals}, consider scalar multiplication $m:E \times [0,1] \to E$ and form the pullback $m^!\mathscr{A}_{\nabla}$, which is canonically identified with $\mathscr{A}_{m^!\nabla}=m^!(A) \oplus m^*(D)$. Again, $\mathfrak{S}(\pr,m,\mathscr{A}_{\nabla})$ is a bundle of Lie algebroids, with $X \times I$ as a Lie algebroid transversal, and
\[\left( X \times I,(i,\mathrm{id})^!m^!(\mathscr{A}_{\nabla}) \right) = (X,i^!(\mathscr{A}_{\nabla})) \times (I,TI) = (X,\mathscr{A}_{i^!\nabla}) \times (I,TI).\]

If $a \in \Gamma(m^!(A))$ is the section in the proof of Theorem \ref{thm : normal forms for transversals}, the local flow $\widetilde{\Phi}_{\epsilon}$ of $(a,0) \in \Gamma(m^!\mathscr{A}_{\nabla})$ is of the form
\[
 \widetilde{\Phi}_{\epsilon} = \left( \begin{array}{cc}
                                       \Phi_{\epsilon} & 0 \\
                                       0 & \Psi_{\epsilon}
                                      \end{array}
 \right)
\]simply because
\[[(a,0),\Gamma(A)]_{\mathscr{A}_{\nabla}} \subset \Gamma(A), \quad [(a,0),\Gamma(D)]_{\mathscr{A}_{\nabla}} \subset \Gamma(D).\]On the other hand, because $D$ is totally intransitive, the local flow $\widetilde{\Phi}_{\epsilon}$ is defined exactly where $\Phi_{\epsilon}$ is defined; in particular, it is defined up to time $\epsilon=1$ on an open neighborhood $U \subset E \times \{0\}$ of $X \times \{0\}$, and at time one it defines an isomorphism of representations
\[
  (\Phi_1,\Psi_1,\varphi_1) : p^!i^!(\nabla)|_U \diffto \nabla|_{\varphi_1(U)},
 \]which restricts to the identity on $i^!(\nabla) \subset p^!i^!(\nabla)$.
\end{proof}

\begin{theorem}[Lie algebroid transversals and cohomology]\label{thm : cohomology concentrated}
For a Lie algebroid transversal $X \subset (M,A)$, and a representation $\nabla:A \curvearrowright D$, there exists an open neighborhood $U$ of $X$ so that the induced map:
\[i^*:\HH^{\bullet}(A|_U;D|_U) \diffto \HH^{\bullet}(i^!(A);i^*(D))\]is an isomorphism.
\end{theorem}
\begin{proof}
By Corollary \ref{cor : nf for reps}, we can find an isomorphism $(\Phi,\Psi,\phi) : p^!i^!(\nabla) \diffto \nabla|_{\phi(U)}$ which restricts to the identity of $i^!(\nabla)$, where $i:X \inc M$ is the inclusion, and $p:U \to X$ is a tubular neighborhood of $X$ in $M$. Hence the composition
\[
 i^!(\nabla) \stackrel{i}{\rmap} \nabla|_{\varphi(U)} \underset{\simeq}{\stackrel{(\Phi,\Psi,\phi)^{-1}}{\rmap}} p^!i^!(\nabla) \stackrel{p}{\rmap} i^!(\nabla)
\]is the identity, hence the composition of
\begin{align*}
\HH(i^!(A);i^*(D)) \stackrel{p^*}{\rmap} & \HH(p^!i^!(A);p^*i^*(D)) \stackrel{\Phi^{-1*}}{\rmap} \HH(A|_{\varphi(U)};D|_{\varphi(U)}) \stackrel{i^*}{\rmap} \HH(i^!(A);i^*(D)) 
\end{align*}
is also identical. On the other hand, if $p:U \to X$ has contractible fibres, which we can always assume to be the case, then $p^*:\HH(i^!(A);i^*(D)) \to \HH(p^!i^!(A);p^*i^*(D))$ is an isomorphism by \cite[Theorem 2]{Cr_VEst}, whence $i^*$ is an isomorphism as well.
\end{proof}

\section{Local triviality and completeness}\label{sec : Local triviality and completeness}

An {\bf Ehresmann connection} for a submersion by Lie algebroids $\mathfrak{S}=(A,p)$ is a linear splitting $A=K(A)\oplus C$. Such a splitting determines, and is determined by, a vector bundle map over $\Sigma$ of \emph{horizontal lift}:
\[
 \mathrm{hor}_C : p^*(TM) \rmap A, \quad \tau_A \circ \mathrm{hor} = \mathrm{id}_{p^*(TM)}.
\]By (the proof of) Lemma \ref{lem : Pullback submersion by Lie algebroids}, if $C$ is an Ehresmann connection for a submersion by Lie algebroids $\mathfrak{S}=(A,p)$, and $\phi:N \to M$ is a smooth map, with induced pullback morphism of submersions by Lie algebroids $(\Phi,\varphi,\phi):\phi^!\mathfrak{S} \to \mathfrak{S}$, then $\phi^!C:=TN \times_{TM}C$ is an Ehresmann connection for $\phi^!\mathfrak{S}$.

An Ehresmann connection $C$ for $\mathfrak{S}$ is called {\bf complete} if either of the following equivalent conditions is met:
\begin{enumerate}[$\mathrm{C}$1)]
 \item the horizontal lift $\mathrm{hor}_C(u) \in \Gamma(C)$ of a complete vector field $u \in \mathfrak{X}(M)$ is a complete section;
 \item the induced Ehresmann connection $H(C):=\varrho_A(C)$ for the underlying submersion by Lie algebroids $\mathfrak{S}(p)=(TE,p)$ is complete;
 \item for every curve $c:[0,1] \to M$, the flow of any time-dependent section $a_{\epsilon}$ with $\hor_C(\dot{c}(\epsilon))=a_{\epsilon}(c(\epsilon))$ defines a {\bf parallel transport} along $c$, i.e., a smooth path of isomorphisms of Lie algebroids
\[
 \Phi^a_{\epsilon}:i_{c(0)}^!\mathfrak{S} \diffto i_{c(\epsilon)}^!\mathfrak{S}, \quad \epsilon \in [0,1],
\]where $i_x$ denotes the inclusion of a point $x \in M$.
\end{enumerate}

Evidently, while every submersion by Lie algebroids admits an Ehresmann connection, not every such submersion admits a \emph{complete} Ehresmann connection:

\begin{example}
 Let $\Sigma = \mathbb{R}^2 \times \R$, $M=\R$, and $p\Sigma \to M$ the canonical projection. Let also $B$ be the cotangent Lie algebroid of a bivector $\pi$ on $N = \mathbb{R}^2$, which is supported on the unit disk, and is nondegenerate at the origin. Then $\mathfrak{S}(m,B)$ is a submersion by Lie algebroids, in which the only transitive fibre is the one over zero. This implies that a complete Ehresmann connection for $\mathfrak{S}$ cannot exist. By constrast, a complete Ehresmann connection for $\mathfrak{S}(m,TN)$ exists.
\end{example}

\begin{proposition}\label{pro : pullback of complete}
Let $C$ be an Ehresmann connection for a submersion by Lie algebroids $\mathfrak{S}=(A,p)$ on $M$, let $\phi:N \to M$ be a smooth map, and let $\phi^!(C)$ be the pullback connection for $\phi^!\mathfrak{S}$.
\begin{enumerate}[a)]
 \item $\phi^!(C)$ complete if $C$ is complete;
 \item $C$ is complete if $\phi^!(C)$ is complete, and $\phi$ is \'etale and surjective;
 \item If $\mathfrak{S}$ has a complete Ehresmann connection and $\phi:N \times [0,1] \to M$ is a smooth map, and $\phi_0:=\phi|_{N \times \{0\}}$, there is an isomorphism of submersions by Lie algebroids
\[
 (\Psi,\psi,\id):\phi_0^!\mathfrak{S} \times TI \diffto \phi^!\mathfrak{S},
\]which restricts to the identity on $\phi_0^!\mathfrak{S}$;
 \item A sufficient condition ensuring that an Ehresmann connection $C$ for $\mathfrak{S}$ be complete is that there exist:
 \begin{itemize}
  \item an open cover $\mathscr{U}=(U_i)$ of $M$,
  \item a local trivialization $\psi:\coprod U_i \times X_i \diffto q_{\mathscr{U}}^*(E)$, and
  \item exhaustions\footnote{A smooth manifold is \emph{exhausted} by a sequence $(X_n)$ of compact, codimension-zero submanifolds with boundary $X_n \subset X$ if $X = \cup X_n$.} $\mathscr{X}_i=(X_{i,n})$ of $X_i$ by compact submanifolds, 
 \end{itemize}such that $q_{\mathscr{U}}^!C$ is tangent\footnote{An Ehresmann connection $C$ for $\mathfrak{S}$ is \emph{tangent} to a submanifold $\Sigma' \subset \Sigma$ if $H(C)_{x} \subset T_{x}\Sigma'$ for all $x \in \Sigma'$. It is tangent to a compact exhaustion $(\Sigma_n)$ of $\Sigma$ if it is tangent to each $\Sigma_n$.} to each $U_i \times X_{i,n}$.
\end{enumerate}
\end{proposition}
\begin{proof}
Let $j_x$ and $i_y$ denote the inclusions $x \in N$ and $y \in M$. Then there is a canonical identification $ \psi_x : j_x^!\phi^!\mathfrak{S} \diffto i_{\phi(x)}^!\mathfrak{S}$; hence if $c : [0,1] \to N$ is a smooth curve, we have that the local flows $\widetilde{\Phi}_{\epsilon}$ of $\hor_{\phi^!C}(\dot{c})$ and $\Phi_{\epsilon}$ of $\hor_{C}(\phi_*\dot{c})$ intertwine these maps:
\[
 \xymatrix{
 (\phi j_{c(0)})^!\mathfrak{S} \ar[d]_{\psi_{c(0)}} \ar[r]^{\widetilde{\Phi}_{\epsilon}} & (\phi j_{c(\epsilon)})^!\mathfrak{S} \ar[d]^{\psi_{c(\epsilon)}}\\
 i_{\phi c(0)}^!\mathfrak{S} \ar[r]_{\Phi_{\epsilon}} & i_{\phi c(\epsilon)}^!\mathfrak{S}
 }
\]This implies that $\widetilde{\Phi}$ exists for all times if $C$ is complete; conversely, if $\phi^!C$ is complete, then $\Phi_{\epsilon}$ exists along curves in $M$ which can be lifted locally to curves in $N$. This proves a) and b).

%
%
%

For c), we claim that the flow $\Phi_{\epsilon}:\phi^!\mathfrak{S} \to \phi^!\mathfrak{S}$ of $a = \hor_{\phi^!(C)}(\bd_{\epsilon})$ assembles into a morphism of Lie algebroids
\[
 (\Phi,\phi) : \phi^!\mathfrak{S} \times T\R \rmap \phi^!\mathfrak{S}
\](extending $\phi$ to a smooth map $\phi:M \times \R \to M$) by declaring $\bd_{\epsilon} \sim a$. Indeed, it suffices to check that
\[\beta \sim b \quad \Longrightarrow \quad [\bd_{\epsilon},\beta]_{A \times T\R} \sim [a,b]_A,\]and for that, it suffices to assume that $\beta=\Phi_{\epsilon}^*(b)$, in which case it follows from
\[
 [\bd_{\epsilon},\beta]_{A \times T\R} = \bd_{\epsilon}\beta = \bd_{\epsilon}\Phi_{\epsilon}^*(b) = \Phi_{\epsilon}^*([a,b]_A).
\]The restriction of $\Phi$ to $\phi_0^!\mathfrak{S} \times TI$ then defines an isomorphism
\[
 (\Phi,\phi) : \phi_0^!\mathfrak{S} \times TI \diffto \phi^!\mathfrak{S}.
\]For d), note that the tangency condition on $q_{\mathscr{U}}^!C$ implies that it is complete (see \cite[Lemma 2]{Matias}), and since $q_{\mathscr{U}}$ is \'etale and surjective, item b) applies. 
\end{proof}

As pointed out in \cite{Matias}, the condition in item d) of Proposition \ref{pro : pullback of complete} is suitable for globalization of local complete connections:

\begin{proposition}\label{pro : subexhaustions}
A complete Ehresmann connection for a submersion by Lie algebroids $\mathfrak{S}$ exists, provided that there exist: 
\begin{enumerate}[1)]
 \item an open cover $\mathscr{U}=(U_i)$ of $M$ and a local trivialization $\psi: \coprod U_i \times X_i \diffto q_{\mathscr{U}}^*(\Sigma)$,
 \item exhaustions $\mathscr{X}_i=(X_{i,n})$ of $X_i$ by compact submanifolds, and
 \item an Ehresmann connection $C$ for $q_{\mathscr{U}}^!\mathfrak{S}$,
\end{enumerate}with the property that the pullback connection of $C$ by $\psi$ is tangent to each $U_i \times X_{i,n}$. 
\end{proposition}
\begin{proof}
Let $\mathscr{U}$, $\psi$, $\mathscr{X}_i$ and $C$ be as in the statement. Let us write $\psi=\coprod \psi_i$, and note that we can assume without loss of generality that $\mathscr{U}=(U_i)$ is locally finite, and has countably many elements. Choose a cover $\mathscr{K}=(K_i)$ of $M$ by compact subsets satisfying $K_i \subset U_i$. Assume for the moment that
\begin{equation}\tag{$\ast$}
 \psi_i(K_i \cap K_j \times \bd X_{i,n}) \cap \psi_j(K_i \cap K_j \times \bd X_{j,m}) \neq \varnothing \quad \Longrightarrow \quad i=j.
\end{equation}This condition ensures that the restriction
\[
 \psi|_{\Upsilon}:\Upsilon \rmap \Sigma, \quad \Upsilon:=\coprod_{i,n}K_i \times \bd X_{i,n}
\]is an embedding, and hence extends to an open embedding $\psi|_{W}:W \to \Sigma$, where $W=\coprod_{i,n}U'_i \times \bd X_{i,n}$ and $U_i'$ is an open neighborhood of $K_i$ in $U_i$. Then $\psi$ pushes $C|_{W}$ to an Ehresmann connection for $\mathfrak{S}|_{\psi(W)}$, and the latter can be extended to an Ehresmann connection $C'$ for $\mathfrak{S}$, simply because Ehresmann connections are sections of an affine bundle. By construction, a such connection $C'$ is automatically complete by item d) of Proposition \ref{pro : pullback of complete}.

It therefore suffices to show that one can choose subexhaustions $\mathscr{Y}_i=(X_{i,\alpha_i(n)})$ of $\mathscr{X}_i$ for which ($\ast$) holds true. Denote the transition maps $\psi_j|_{U_i\cap U_j \times X}^{-1} \circ \psi_i|_{U_i\cap U_j \times X}$ by $\psi_{ji}$. We split the proof into three cases.
\vspace{0.1cm}

\noindent\emph{Case One: $\mathscr{U}=\{U_1,U_2\}$.} In the case where the open cover $\mathscr{U}$ consists of two open sets $U_1$ and $U_2$, we construct subexhaustions $\mathscr{Y}_1=(X_{1,\alpha_1(n)})$ and $\mathscr{Y}_2=(X_{2,\alpha_2(n)})$ by induction as follows: the initial step sets $\alpha_1(1):=1$ and
 \[\alpha_2(1) = \min \{ k \ | \ \psi_{21}(K_1 \cap K_2 \times X_{1,1}) \subset K_1 \cap K_2 \times \mathrm{int}(X_{2,k}) \}.\]Suppose the values of $\alpha_1(n),\alpha_2(n)$ have been chosen for all $n \leqslant r$ in such a way that
 \begin{align*}
  \psi_{21}(K_1 \cap K_2 \times X_{1,\alpha_1(n)}) \subset K_1 \cap K_2 \times \mathrm{int}(X_{2,\alpha_2(n)}) \\
  \psi_{12}(K_1 \cap K_2 \times X_{2,\alpha_2(n-1)}) \subset K_1 \cap K_2 \times \mathrm{int}(X_{1,\alpha_1(n)})
 \end{align*}Then these relations are satisfied up to $r+1$ if we set
\begin{align*}
 \alpha_1(r+1)&:= \min\{k \ | \ \psi_{12}(K_1 \cap K_2 \times X_{2,\alpha_2(r)}) \subset K_1 \cap K_2 \times \mathrm{int}(X_{1,k})\}, \\
 \alpha_2(r+1)&:=\min\{k \ | \ \psi_{21}(K_1 \cap K_2 \times X_{1,\alpha_1(r+1)}) \subset K_1 \cap K_2 \times \mathrm{int}(X_{2,k})\}.
\end{align*}This algorithm constructs increasing maps $\alpha_i:\N \to \N$ for which the subexhaustions $\mathscr{Y}_i:=(X_{i,\alpha_i(n)})$ satisfy the condition in the statement.
\vspace{0.1cm}

\noindent\emph{Case Two: $\mathscr{U}=\{U_1,U_2,...,U_{N}\}$.} In this case, the desired subexhaustions $\mathscr{Y}_i$ are constructed by ${N \choose 2}$ successive applications of Case One, to the following lexicographic sequence of pairs of open sets:
\begin{align*}
 \{U_1,U_2\}, \  \hdots, \ \{U_1 ,U_{N}\}, \ \{U_2 ,U_3\}, \ \hdots, \ \{U_2 ,U_{N}\}, \ \hdots, \ \{U_{N-1},U_{N}\}.
\end{align*}

\vspace{0.1cm}

\noindent\emph{Case Three: Countable $\mathscr{U}$.} Because $\mathscr{U}=(U_i)_{i \in \N}$ is locally finite, the set
\[
 \varLambda(i):=\{j \in \Lambda(i) \ | \ U_i \cap U_j \neq \varnothing, \ j \geqslant i\}
\]is finite for each $i$. Write $M_i:=\bigcup_{j \in \varLambda(i)}U_j$. We define subexhaustions $\mathscr{Y}_i:=(X_{i,\alpha_i(n)})$ by inductively applying Case Two to $\left(M_{r},(U_i)_{i \in \varLambda(r)}\right)$. This process generates the desired compact subexhaustions $\mathscr{Y}_i$.\qedhere
\end{proof}
 
\begin{definition}
 A submersion by Lie algebroids $\mathfrak{S}=(A,p)$ is called {\bf locally trivial} if every point $x \in M$ has an open neighborhood $U \subset M$, over which there is an isomorphism $(\Psi,\psi,\id):(TU,\id) \times i_x^!\mathfrak{S} \diffto i_U^!\mathfrak{S}$.
\end{definition}

\begin{theorem}\label{thm : LAB iff}
The following are equivalent conditions on a submersion by Lie algebroids $\mathfrak{S}$:
\begin{enumerate}[i)]
\item It admits a complete Ehresmann connection;
 \item It is locally trivial.
\end{enumerate}
\end{theorem}
\noindent Hence locally trivial submersions by Lie algebroids in our sense are exactly what \cite[Definition 1.1]{BraZhu} calls \emph{fibrations} over $TM$.
\begin{proof}
Let $\mathfrak{S}=(A,p)$ be a submersion by Lie algebroids on $M$.

\noindent {\bf i) $\Rightarrow$ ii)} If $\mathfrak{S}$ admits a complete Ehresmann connection, choose an open cover $\mathscr{U}=(U_i)$ of $M$, for which there exists smooth contractions:
\[
 f_i:U_i \times [0,1] \rmap U_i, \quad f_{i,0}(x)=x_i, \quad f_{i,1}(x)=x, \quad x \in U_i,
\]and consider the ensuing homotopy $f:\coprod U_i \times [0,1] \to M$. By item c) of Proposition \ref{pro : pullback of complete}, we have an isomorphism $f_0^!\mathfrak{S} \times TI \diffto f^!\mathfrak{S}$, and in particular, an isomorphism between
\[
f_0^!\mathfrak{S} = \coprod TU_i \times i_{x_i}^!\mathfrak{S}, \quad \text{and} \quad f_1^!\mathfrak{S} = q_{\mathscr{U}}^!(\mathfrak{S})=\coprod i_{U_i}^!\mathfrak{S}.
\]

\noindent {\bf ii) $\Rightarrow$ i)} Choose an open cover $\mathscr{U}=(U_i)$ of $M$, over which there exists a Lie algebroid isomorphism
\begin{align*}
   (\Psi,\psi,\id):\coprod (TU_i,\id) \times i_{x_i}^!\mathfrak{S} \diffto q_{\mathscr{U}}^!\mathfrak{S},
  \end{align*}where $x_i \in U_i$. Note that the product Ehresmann connection $\coprod TU_i \times X_i$ for $q_{\mathscr{U}}^!\mathfrak{S}$ satisfies the hypotheses of Proposition \ref{pro : subexhaustions} with respect to any exhaustions $\mathscr{X}_i=(X_{i,n})$ of $X$ by compact submanifolds, and therefore a complete connection $C$ exists for $\mathfrak{S}$.
\end{proof}

\section{The system of local coefficients}\label{sec : The system of local coefficients}

Given a submersion by Lie algebroids $\mathfrak{S}=(A,p)$ and a representation $\nabla:A \acts D$, there is an induced presheaf of graded vector spaces
\[
 \mathscr{H}_{\nabla}:\mathrm{Open}(M) \rmap \mathrm{GVect}, \quad \mathscr{H}_{\nabla}(U):=\HH(A|_{\Sigma|_U};D|_{\Sigma|_U}),
\]whose associated sheaf we denote by $\pr:\mathscr{H}_{\nabla}^+\to M$.

\begin{lemma}\label{lem : cohomology presheaf locally constant for lab}
 The presheaf $\mathscr{H}_{\nabla}$ is locally constant if $\mathfrak{S}$ is a locally trivial submersion by Lie algebroids.
\end{lemma}
\begin{proof}
First note that if $\mathfrak{S}=(A,p)$ is a locally trivial submersion by Lie algebroids, and $\nabla:A \acts D$ is a representation, then $\mathfrak{S}_{\nabla}=(\mathscr{A}_{\nabla},p)$ is again a locally trivial submersion by Lie algebroids. Indeed, if $C$ is a complete Ehresmann connection for $\mathfrak{S}$, then $C_{\nabla}=C\oplus D$ is a complete Ehresmann connection for $\mathfrak{S}_{\nabla}$; hence Theorem \ref{thm : LAB iff} implies that $\mathfrak{S}_{\nabla}$ is locally trivial. So we are reduced to showing that, for a contractible open neihborhood $U$ of $x$, and $\pr:U \times i_x^*(\Sigma) \to i_x^*(\Sigma)$ the canonical projection, $\mathscr{H}_{\pr^!i_x^!\nabla}$ is locally constant, and the latter is a consequence of \cite[Theorem 2]{Cr_VEst}, which implies that $\pr^*:\mathrm{H}(i_x^!\nabla) \to \mathrm{H}(\pr^!i_x^!\nabla)$ is an isomorphism, whence $\mathscr{H}_{\pr^!i_x^!\nabla} \simeq U \times \mathrm{H}(i_x^!\nabla)$, where $\mathrm{H}(i_x^!\nabla)$ has the discrete topology.
\end{proof}

In particular, for a locally trivial submersion by Lie algebroids $\mathfrak{S}=(A,p)$, $\pr:\mathscr{H}_{\nabla}^+\to M$ is a covering space, the fibre over any $x \in M$ being canonically identified with $\HH^q(A_x;D_x)$. We describe next the monodromy of $\pr:\mathscr{H}_{\nabla}^+\to M$, i.e., the homomorphism of abelian groups
\[
 \mathrm{mon}(c):\HH^q(i_{c(0)}^!\mathfrak{S}) \rmap \HH^q(i_{c(1)}^!\mathfrak{S})
\]induced by a curve $c:[0,1] \to M$. By item c) in Proposition \ref{pro : pullback of complete}, $i_0^*:\HH^q(c^!\mathfrak{S}) \diffto \HH^q(i_{c(0)}^!\mathfrak{S})$ is an isomorphism, and similarly for $i_1^*:\HH^q(c^!\mathfrak{S}) \diffto \HH^q(i_{c(1)}^!\mathfrak{S})$. Hence we set
\[
 \mathrm{mon}(c):=(i_{c(1)})^* \circ (i_{c(0)})^{*-1}.
\]

\begin{theorem}[Monodromy]\label{thm : Monodromy}
Let $\mathfrak{S}=(A,p)$ be a submersion by Lie algebroids, endowed with a complete Ehresmann connection $C$ and a representation $\nabla:A \acts D$. For a curve $c:[0,1] \to M$, denote by $\Phi_{\epsilon}$ parallel transport of $\bd_{\epsilon}$ with respect to $c^!(C)$. Then
\[
 \mathrm{mon}(c)= \Phi_1^{\ast-1}.
\]
\end{theorem}
\begin{proof}
Let $a = \hor_{c^!(C)}(\bd_{\epsilon}) \in \Gamma(c^!\mathfrak{S})$, and recall that its flow induces an isomorphism
\[
 (\Phi,\phi):c_0^!\mathfrak{S} \times TI \diffto c^!\mathfrak{S}.
\]Because $a$ spans $c^!(C)$ and $\bd_{\epsilon} \sim a$, the induced Ehresmann connection $(\Phi,\phi)^!C$ for $c_0^!\mathfrak{S} \times TI$ is the flat one.
\end{proof}

\begin{remark}
For transitive Lie algebroids, the system of local coefficients $\mathscr{H}_{\nabla}$ had first appeared in \cite{ChenLiu_Loc}, and was described in great detail in \cite{BruzzoMencattini}. 
\end{remark}

\begin{corollary}[Gau\ss-Manin connection]
 If $\HH(A_x;D_x)$ is finite-dimensional for all $x \in M$, there exist a graded vector bundle $\mathbb{H} \to M$, and a degree-preserving flat connection $\nabla^{\mathrm{GM}}:TM \acts \mathbb{H}$, with the property that
 \begin{enumerate}[a)]
  \item $\mathbb{H}_x=\HH(A_x;D_x)$ for all $x \in M$;
  \item local sections of $\pr:\mathscr{H}_{\nabla}^+ \to M$ are exactly the local horizontal sections of $\nabla^{\mathrm{GM}}$.
 \end{enumerate}
\end{corollary}
\begin{proof}
Since $\mathscr{H}_{\nabla}^+$ is a locally finitely generated $\R$-module by hypothesis, $\mathscr{H}_{\nabla}^+ \otimes_{\R}C^{\8}(M)$ is a locally finitely generated $C^{\8}(M)$-module, and hence corresponds by the Serre-Swan theorem to the sheaf of local sections of a smooth vector bundle $\mathbb{H}$ on $M$. The canonical map $\mathscr{H}_{\nabla}^+ \to \mathbb{H}$ is a smooth bijection. Since $\mathscr{H}_{\nabla}^+$ is a covering space, it defines a flat connection by declaring a local section of $\mathbb{H}$ to be flat if it comes from a local section of $\mathscr{H}_{\nabla}^+$.
\end{proof}

The discussion so far allows us to describe a version of homotopy invariance for Lie algebroid cohomology as a twisting by the system of local coefficients:

\begin{corollary}\label{cor : Homotopy invariance}
Let $A$ be a Lie algebroid on $M$, $\nabla:A \acts D$ a representation, and let $\phi : N \times [0,1] \to M$ be a smooth map for which each $\phi_{\epsilon}:N \to M$ is transverse to $A$. If $\phi^!(A)$ turn $\pr: N \times [0,1] \to [0,1]$ into a locally trivial submersion by Lie algebroids, then $\phi_{\epsilon}^*(\omega)$, $\omega \in \mathrm{H}^q(A)$, is parallel transport of $\phi_{0}^*(\omega)$ with respect to any complete connection.
\end{corollary}

When the system of local coefficients is trivial, this reduces to the familiar statement that $\phi_{\epsilon}^*(\omega)$ does not depend on $\epsilon$; in the general case, it describes parallel transport $\Phi_{\epsilon}^{*-1}\phi_{0}^*(\omega)$. (See also the forthcoming work \cite{BraRoss}). The appearance of the twisting by the local system of coefficients is explained by the fact that, in the generality of Corollary \ref{cor : Homotopy invariance}, for each fixed $x \in N$, the paths $\phi_{\epsilon}(x) \in M$ need \emph{not} lie in the same leaf of $A$ -- a phenomenon which does not occur in classical de Rham theory, and a trivial example of which we present below:

\begin{example}
 Let $A$ be a Lie algebroid on $M$ with trivial bracket and anchor (i.e., a vector bundle). Then every section of $\wedge A^*$ is a cocycle. If $\phi_{\epsilon}(x)$ is the flow of a vector field $v \in \X(M)$, then $\phi_{\epsilon}(x)$ lies in a leaf of $A$ exactly when $v$ vanishes at $x$, and $\phi_{\epsilon}^*(\omega) \neq \phi_0^*(\omega)$ for some $\epsilon$, unless $\mathscr{L}_v\omega(x) = 0$.
\end{example}

\section{Localization}\label{sec : Localization}

Let $A$ be a Lie algebroid on $M$, and $\nabla:A \acts D$ a representation. Fix an open cover $\mathscr{U}=(U_i)_{i \in I}$ of $M$, and for each finite subset $\alpha \subset I$, denote by $U_{\alpha}$ the intersection of all its members $\cap_{i \in \alpha}U_i$. Consider the \v{C}ech-de Rham double complex of $\mathscr{U}$ with values in the sheaf $\Omega_{\nabla}$ of smooth section of $\wedge A^* \otimes D$:
\[
 C^p(\mathscr{U};\Omega^q_{\nabla}):= \prod_{|\alpha|=p+1} \Omega^q_{\nabla}(U_{\alpha}), \quad p,q \geqslant 0,
\]with differentials $\dd_{\nabla} : C^p(\mathscr{U};\Omega^q_{\nabla}) \to C^p(\mathscr{U};\Omega^{q+1}_{\nabla})$ induced by $(\Omega_{\nabla},\dd_{\nabla})$, and the \v{C}ech differential,
\begin{align*}
 \delta : C^p(\mathscr{U};\Omega^q_{\nabla}) \to C^{p+1}(\mathscr{U};\Omega^q_{\nabla}), \quad (\delta\omega)_{\alpha_0\cdots\alpha_p} = \sum_{i=0}^p (-1)^i\omega_{\alpha_0\cdots\widehat{\alpha_i}\cdots\alpha_p}
\end{align*}
The natural map $r: \Omega(A;D) \rmap C^0(\mathscr{U};\Omega^q_{\nabla})$, $r(\omega)_i:=\omega|_{U_i}$ makes the augmented complex
\[
 0 \rmap \Omega(A;D) \stackrel{r}{\rmap} C^0(\mathscr{U};\Omega^q_{\nabla}) \stackrel{\delta}{\rmap} C^1(\mathscr{U};\Omega^q_{\nabla}) \stackrel{\delta}{\rmap} \cdots \stackrel{\delta}{\rmap} C^p(\mathscr{U};\Omega^q_{\nabla}) \stackrel{\delta}{\rmap}\cdots
\]exact, and so may we deduce as in \cite[Section 2.9]{BottTu} that
\begin{lemma}[Mayer-Vietoris principle]\label{lem : Mayer-Vietoris principle}
 The complex $(\Omega_{\nabla},\dd_{\nabla})$ is quasi-isomorphic to the \emph{total complex} $(\mathbf{C}_{\nabla},\mathbf{d})$,
\[
 \mathbf{C}_{\nabla}^n:= \oplus_{p+q=n}C^p(\mathscr{U};\Omega^q_{\nabla}), \quad \mathbf{d}:=\delta + (-1)^p\dd_{\nabla}.
\]
\end{lemma}\noindent See in this connection the recent work \cite{Bruzzo}.

\begin{proposition}[Leray principle for Lie algebroid cohomology]\label{pro : Leray}
 If $\mathfrak{S}=(A,p)$ is a Lie algebroid bundle, and $\mathscr{U}=(U_i)$ is a good cover of $M$, there is a bigraded spectral sequence $\dd^r : E_r^{p,q} \to E_r^{p+r,q-r+1}$, which converges to $\HH(A;D)$, and whose second page is given by \v{C}ech cohomology with values in the presheaves $\mathscr{H}^q_{\nabla}$
\[
 E_2^{p,q}:=\HH^p(\mathscr{U};\mathscr{H}^q_{\nabla}) \Longrightarrow \HH^{p+q}(A;D).
\]
\end{proposition}
\begin{proof}
 Applying \cite[Theorem 4.14]{BottTu} to the double complex $(C^p(\mathscr{U};\Omega^q_{\nabla}),\dd_{\nabla}, \delta)$ we get as in \cite[Theorem 4.18]{BottTu} a spectral sequence converging to the cohomology of the total complex $(\mathbf{C}_{\nabla},\mathbf{d})$, whose first and second pages are given by
 \[
  E_1^{p,q}=C^p(\mathscr{U};\mathscr{H}_{\nabla}^q), \qquad E_2^{p,q}=\HH^p(\mathscr{U};\mathscr{H}_{\nabla}^q),
 \]while Lemma \ref{lem : Mayer-Vietoris principle} identifies $\HH(\mathbf{C}_{\nabla})$ with $\HH(A;D)$.
\end{proof}
%

\begin{theorem}\label{thm : localization of cohomology}
 Let $\mathfrak{S}=(A,p)$ be a locally trivial submersion by Lie algebroids, and $\nabla:A \acts D$ a representation. Then for any point $x \in M$, the homomorphism
  \[
  i_x^* : \mathrm{H}^{n}(A;D) \rmap \mathrm{H}^{n}(A_x;D_x)
 \]induced by inclusion is injective, provided that:
 \begin{tasks}(2)
  \task $\mathrm{H}^{q}(A_x;D_x)=0$ for each $q<n-1$,
  \task $M$ is connected,
 \end{tasks}and either of the following conditions holds:

 \begin{tasks}[counter-format = $c$tsk[1])$ $](2)
  \task $\phantom{1}\mathrm{H}^{n-1}(A_x;D_x)=0$, or
  \task $\phantom{1}\pi_1(M,x)=0$ and $\dim\mathrm{H}^{n-1}(A_x;D_x)<\8$.
 \end{tasks}
\end{theorem}
\begin{proof}
Assumption a) implies that the spectral sequence of Proposition \ref{pro : Leray} satisfies $E^{n-q,q}_2=0$ for all $q<n-1$. This implies on the one hand that $E^{n-q,q}_{\8}=0$ for all such $q<n-1$, and that $E^{n-q,q}_r$ are subspaces of $E^{n-q,q}_2$ for all $r \geqslant 2$ and $q \geqslant n-1$. Assuming either c1) or c2) implies that also $E^{1,n-1}_{\8}=0$, and hence that the filtration of $\mathrm{H}^n(A;D)$ has a single nontrivial step $E^{0,n}_{\8}=\mathrm{H}^0(M;\mathscr{H}^{n}_{\nabla})$. Indeed, assuming c1) this follows from $E^{1,n-1}_2=0$, while c2) implies that $\mathrm{Et}(\mathscr{H}^{n-1}_{\nabla}) \simeq M \times \mathscr{H}^{n-1}(A_x;D_x)$, and that
\[
 E^{1,n-1}_{\8}=\mathrm{H}^1(M;\mathscr{H}^{n-1}_{\nabla}) \simeq \mathrm{H}^1(M) \otimes \mathrm{H}^{n-1}(A_x;D_x)  =0.
\]Hence\[s_{\cdot}:\mathrm{H}^{n}(A;D) \diffto \mathrm{H}^0(M;\mathscr{H}^{n}_{\nabla}), \quad s_{\omega}(y):=i_y^*[\omega], \quad \omega \in \Omega^q(A)\]is an isomorphism. Now observe that $i_x^*:\mathrm{H}(A;D) \rmap \mathrm{H}(A_x;D_x)$ factors as
\[
 \xymatrix{
 \mathrm{H}^q(A;D) \ar[dr]_{i_x^*} \ar[rr]^{s_{\cdot}} && \mathrm{H}^0(M;\mathscr{H}^{n}_{\nabla}) \ar[dl]^{\mathrm{ev}_x} \\
 & \mathrm{H}^q(A_x;D_x) &
 }
\]where $\mathrm{ev}_x$ denotes the map $\mathrm{ev}_x(s):=s(x)$ for $s \in \mathrm{H}^0(M;\mathscr{H}^{n}_{\nabla})$, which is injective, since the set of points on which two sections of a covering space agree is both open and closed, and $M$ is assumed connected in b). Therefore $i_x^*$ is injective.
\end{proof}

\end{document}